\documentclass[letterpaper]{amsart}
\usepackage{amsfonts,amssymb,amscd,amsmath,amsthm}
\usepackage[all]{xy}
\usepackage[dvips]{graphicx}
\usepackage{color}

\newtheorem{Thm}{Theorem}[section]
\newtheorem{Lem}[Thm]{Lemma}

\theoremstyle{definition}

\newtheorem{Def}[Thm]{Definition}
\newtheorem{Asm}[Thm]{Assumption}

\DeclareMathOperator{\dom}{dom}
\DeclareMathOperator{\pos}{pos}
\DeclareMathOperator{\stem}{stem}
\newcommand{\fmax}{f_\text{max}}  
\newcommand{\gmin}{g_\text{min}}  
\newcommand{\fnto}{\ensuremath{\rightarrow}}  
\newcommand{\al}[1]{\ensuremath{{\aleph_{#1}}} }          
\newcommand{\om}[1]{\ensuremath{{\omega_{#1}}} }          

\newcommand{\esm}{\ensuremath{\prec}}  

\newcommand{\std}[1]{\ensuremath{\check{#1}}}              

\newcommand{\incomp}{\ensuremath{\perp}}
\newcommand{\comp}{\ensuremath{\parallel}}

\DeclareMathOperator{\trunk}{trunk}

\newcommand{\myc}{c^{\exists}}
\newcommand{\mycfa}{c^{\forall}}

\begin{document}
\date{\today}

\title{Even more simple cardinal invariants}
\author[Jakob Kellner]{Jakob Kellner}
\address{Kurt G\"odel Research Center for Mathematical Logic\\
 Universit\"at Wien\\
 W\"ahringer Stra\ss e 25\\
 1090 Wien, Austria}
\curraddr{Einstein Institute of Mathematics\\
  Edmond J. Safra Campus, Givat Ram\\
  The Hebrew University of Jerusalem\\
  Jerusalem, 91904, Israel}
\email{kellner@fsmat.at}
\urladdr{http://www.logic.univie.ac.at/$\sim$kellner}
\thanks{Supported by a European Union Marie Curie EIF fellowship, contract MEIF-CT-2006-024483.}


\begin{abstract}
  Using GCH, we force the following: There are continuum many 
  simple cardinal characteristics with pairwise different values.
\end{abstract}
\maketitle

\section{Introduction}
The union of countably many Lebesgue nullsets is again a nullset. On the other
hand, there are $2^\al0$ many nullsets with non-null union.  If we assume
$\lnot$CH, i.e.\ $2^\al0>\al1$, then it is interesting to ask about the minimal
size of a family of nullsets with non-null union.  This is a cardinal number
between (including) $\al1$ and $2^\al0$.  Such cardinal numbers (or their
definitions) are called cardinal characteristics.

There are numerous examples of such characteristics using notions from
measure theory, topology or combinatorics.  If $a$ and $b$ are such
characteristics, on can learn something about the underlying
notions by either proving dependencies (e.g. $a\leq b$) in ZFC, or by showing
that $a$ and $b$ are independent (usually by finding forcing notions $P$ and
$Q$ such that $P$ forces $a<b$ and $Q$ forces $b<a$, or by using MA).

Blass~\cite{MR1234278} introduced a classification of cardinal
characteristics, and in particular defined $\Pi^0_1$ characteristics.
Goldstern and Shelah~\cite{MR1201650} showed that there are many
$\Pi^0_1$ characteristics. In particular:
\begin{quote}
  Assume CH.\@ Assume that $\kappa_\epsilon^\al0=\kappa_\epsilon$ for all
  $\epsilon\in \om1$ and that the
  functions $f_\epsilon,g_\epsilon:\omega\to\omega$ ($\epsilon\in \om1$)
  are sufficiently different. Then there is a partial order $P$
  preserving cardinals which forces that
  $\mycfa(f_\epsilon,g_\epsilon)=\kappa_\epsilon$  
  for all $\epsilon\in\om1$.
\end{quote}
(The $\Pi^0_1$ cardinal characteristics $\mycfa(f,g)$ are defined in
\ref{def:cardchar}.)

If the $\kappa_\epsilon$ are pairwise different, 
then in the forcing extension the size of the continuum is at least
$\aleph_{\om1}$. So $\al1$, the number of different characteristics
in the forcing extension, is smaller than the continuum. 

In this paper, we assume GCH in the ground model and 
modify the construction to get a universe satisfying:
\begin{quote}
  There are continuum many pairwise different cardinal characteristics of
  the form $\mycfa(f_\epsilon,g_\epsilon)$.
\end{quote}

We give a relatively simple proof for this result. A slightly stronger result
was promised in~\cite{MR1201650} to appear in a paper called 448a, which never
materialized: a ``perfect set'' of pairwise different characteristics. Shelah
and the author are working on new creature forcing iteration techniques.  One
of the applications will hopefully be a proof of the perfect set result, as
well as similar results for the dual notions $\myc$ (which require lim-inf
constructions, cf.~\cite{KrSh:872}). All these constructions are considerably
more difficult than the ones in this paper.

\section{The theorem and the forcing}

\begin{Def}\label{def:cardchar}
  Let $f,g:\omega\to \omega\setminus 1$ be such that $f(n)>g(n)$
  for all $n$.
  \begin{itemize}
    \item $B:\omega\fnto \mathfrak{P}(\omega)$ 
      is an $(f,g)$-slalom if $B(n)\subseteq f(n)$
      and $|B(n)|<g(n)$ for all $n\in\omega$.
    \item A family $\mathfrak{B}$ of $(f,g)$-slaloms $\forall$-covers,
      if for all $\nu\in \prod_{n\in\omega}f(n)$ there is a
      $B\in \mathfrak{B}$ such that $\nu(n)\in B(n)$ for all $n\in\omega$.
    \item $\mycfa(f,g)$ is the minimal size of a $\forall$-covering
      family of $(f,g)$-slaloms.
  \end{itemize}
\end{Def}
See~\cite{MR1201650} for more about $\mycfa(f,g)$.
We are going to prove the following:
\begin{Thm}
  Assume that CH holds, that 
  $\mu=\mu^{\al0}$, and for
  $\epsilon\in \mu$, 
  $\kappa_\epsilon<\mu$ is a cardinal such
  that $\kappa_\epsilon^{\al0}=\kappa_\epsilon$.
  Then there is a forcing notion $P$ 
  and there are $P$-names $f_\epsilon,g_\epsilon$ 
  such that $P$ preserves
  cardinals and forces the following:
  $2^{\al0}=\mu$, and $\mycfa(f_\epsilon,g_\epsilon)
  =\kappa_\epsilon$ for all
  $\epsilon\in \mu$.
\end{Thm}

If we assume GCH, we can find such $\mu$ and $\kappa_\epsilon$ such that the
$\kappa_\epsilon$ are pairwise different,\footnote{Let $\mu=\aleph_\mu$ be the
  $\om1$-th iterate of the function $\alpha\mapsto \aleph_\alpha$ (taking the
  union at limits), and pick cardinals $\kappa_\epsilon<\mu$  with uncountable
  cofinality.}
i.e., we get continuum many pairwise different invariants in the extension.

For the rest of the paper we assume that the conditions of the theorem are
satisfied (in the ground model).

We will use $\epsilon,\epsilon',\epsilon_1,\dots$ for elements of
$\mu$.

\begin{Asm}
  $(g_{n,l})_{n\in \omega,0\leq l<2^n}$ and
  $(f_{n,l})_{n\in \omega,-1\leq l<2^n}$
  are sufficiently fast
  growing sequences of natural numbers,
  such that $0=f_{0,-1}$, 
  $f_{n+1,-1}=f_{n,2^n-1}$
  and $f_{n,l-1}\ll g_{n,l}\ll f_{n,l}$. 
  We set $\fmax(m)=f_{m,2^m-1}$ and 
  $\gmin(m)=g_{m,0}$. 
\end{Asm}
Sufficiently fast growing means the following:\footnote{The second inequality
  guarantees that there is a $g$-big norm (cf.\ \ref{lem:bignorm}), 
  and the first one is extracted from
  the proof of~\ref{lem2}. Obviously one can try to find weaker
  conditions, but we do not try to find optimal bounds in this paper.}
$g_{n,l}>2\cdot f_{n,l-1}^{n\cdot \fmax(n-1)^n}$%
, and
$f_{n,l}> g_{n,l}^{n+1}$%
. ($\fmax(n-1)^n$ denotes the $n$-th power of $\fmax(n-1)$.)

We identify $[0,2^n-1]$ with the set of binary sequences of
length $n$, ordered lexicographically.
So for $s\in 2^n$, we can define $f_s=f_{n,s}$ and $g_s=g_{n,s}$.
If $\eta\in 2^\omega$, then we can define $f:\omega\to\omega$ by
$f(n)=f_{\eta\restriction n}$, and $g$ analogously. 

We will define $P$ so that $P$ adds Sacks generics $\eta_\epsilon$
($\epsilon\in\mu$) and forces that
$\mycfa(f_\epsilon,g_\epsilon)=\kappa_\epsilon$ for the
$(f_\epsilon,g_\epsilon)$ defined by $\eta_\epsilon$.

Fix $s\in 2^n$. If $a$ is a subset of $f_s$ (i.e.\ of 
the interval $[0,f_s-1]$), we set
$\mu_s(a)=\ln_{g_s}(|a|)$. (Alternatively, We could use
any other $g_s$-big norm as well, i.e.\ a norm satisfying
the following:)

\begin{Lem}\label{lem:bignorm}
  $\mu_s: \mathfrak{P}(f_s)\to \mathbb{R}$ satisfies:
  ($a,b\subseteq f_s$)
  \begin{itemize}
    \item If $b\subseteq a$, then $\mu_s(a)\geq \mu_s(b)$.
    \item $\mu_s(f_s)\geq n$.
    \item $\mu_s(\{t\})<1$ for all $t\in f_s$.
    \item If $F$ is a function from $a$ to $g_s$, then
      there is a $b\subseteq a$ such that 
      $F\restriction b$ is constant
      and
      $\mu_s(b)\geq \mu_s(a)-1$.
  \end{itemize}
\end{Lem}
Note that $\mu_s(b)\geq 2$ implies that $|b|>g_s$.

Set $\omega^{\leq n}=\bigcup_{l\leq n}\omega^l$.
We will use trees $T\subseteq \omega^{<\omega}$ (or $2^{<\omega}$ or
$\omega^{\leq n}$). For a node
$s\in T\cap \omega^n$, $n$ is called the height of $s$. A branch $b$ in $T$ is a
maximal chain (i.e.\ a maximal set of pairwise comparable nodes).
We can identify $b$ with an element of $\omega^\omega$ (or $\omega^n$),
and denote with $b\restriction h$ the element of $b$ of height $h$ 
(for all $h<\omega$ or $h<n$, respectively).
A front $F$
in $T$ is a set of pairwise incomparable nodes such that every branch of
$T$ hits a node in $F$. When talking about nodes, we use the terms
``comparable'' and ``compatible'' interchangeably. We use the 
symbol $\incomp$ for incompatible (i.e.\ incomparable, when talking 
about nodes), and we use $\comp$ for
compatible. A  splitting node $s$ is a node with at least two
immediate successors.
The first splitting node is called
$\stem(T)$.

A Sacks condition $T$ is a perfect tree, i.e.\ $T\subseteq 2^{<\omega}$ is such
that for every $s\in T$ there is a splitting node $s'>s$.  Equivalently, along
every branch of $T$ there are infinitely many splitting nodes. So the set of
the $n$-th splitting nodes forms a front.

We will use Sacks conditions as well as other ``lim-sup'' finite splitting tree
forcings. Actually we will use finite approximations to such trees, but it
might be useful to first specify the objects we are approximating:
For $\eta\in 2^\omega$, $T$ is an $\eta$-tree,
if $T\subseteq \omega^{<\omega}$ is a tree without leaves (``dead ends'') such
that $s(n)<f_{\eta\restriction n}$ for all $s\in T$.  For an $\eta$-tree $T$
and $s\in T\cap \omega^n$, we set $\mu_{T}(s)=\mu_s(A)$, where $A$ is the set
of immediate $T$-successors of $s$.  $T$ is fat\label{def:fat} if
$\limsup_{n\to\infty}(\mu_{T}(b\restriction n))=\infty$
for every branch $b$ of $T$. $Q_\eta$
is the partial order of fat trees ordered by inclusion.\footnote{
  $Q_\eta$ is a special case of a lim-sup finite splitting tree forcing $Q$,
  informally defined as follows: 
  $Q$ is defined by a finite splitting tree $T_0$ and a norm on the
  successor sets. $T\subseteq T_0$ is a condition of $Q$
  if for all branches $b$ of $T$, the 
  $T$-norm of $b\restriction n$ gets arbitrarily large.
 
  Sacks forcing is a simple
  example of such a forcing: $T_0$ is $2^{<\omega}$.
  Pick $s\in 2^{<\omega}$ and set $A=\{s^\frown 0,s^\frown 1\}$.
  Then we set $\mu(A)=1$ and $\mu(B)=1$ for all proper subsets $B$ 
  of $A$.}

It is easy to see (and analogous to Sacks forcing) that all forcing notions
$Q_\eta$ are
and $\omega^\omega$-bounding.\footnote{This holds of course for all 
  lim-sup finite splitting tree forcings.}
In~\cite{MR1201650}, Goldstern and Shelah picked $\om1$ many
different $\eta_\epsilon$, defined $P_\epsilon$ to be the countable support
product of $\kappa_\epsilon$ many copies of $Q_{\eta_\epsilon}$, and defined
$P$ to be the countable support product of the $P_\epsilon$. Then $P$ forces
$\mycfa(f_\epsilon,g_\epsilon)=\kappa_\epsilon$.

We need $\mu>2^\al0$ many different $\eta$, so $\eta_\epsilon$ will be a name
 (for a Sacks real). Then we again want to use $\kappa_\epsilon$ many copies of
$Q_{\eta_\epsilon}$.  Instead of using a composition of forcings, we more
explicitly use finite approximations to fat trees:

\begin{Def}
  Assume $s\in 2^n$.
  \begin{itemize}
    \item
      $T$ is an $s$-tree if
      $T\subseteq \omega^{\leq n+1}$ is a tree, 
      every branch has length $n+1$ and
      $t(m)<f_{s\restriction m}$ for each $m\leq n$
      and $t\in T\cap \omega^{m+1}$.
    \item 
      For $m\leq n$ and
      $t\in T\cap \omega^m$, $t$ is an $l$-large splitting node, 
      if $\mu_{s\restriction m}(A)\geq l$ for 
      the set $A$ of immediate $T$-successors of $t$.
    \item
      $T$ has $l$-large splitting if the 
      set of $l$-large splitting nodes forms
      a front.
  \end{itemize}
\end{Def}

\begin{Def}
  \begin{itemize}
    \item
      For every $\epsilon$ in $\mu$, pick 
      some $I_\epsilon$ of size $\kappa_\epsilon$
      such that $\mu$ and all the $I_\epsilon$ are
      pairwise disjoint. Set $I=\mu\cup\bigcup_{\epsilon\in\mu}I_\epsilon$.
    \item
      We define $\varepsilon:I\to I$:
      If $\alpha\in I_{\epsilon}$, then $\varepsilon(\alpha)=\epsilon$.
      If $\epsilon\in\mu$, then $\varepsilon(\epsilon)=\epsilon$.
  \end{itemize}
\end{Def}
$I$ will be the index set of the product forcing.  We will 
use $\alpha,\beta,\dots$ for elements of $I$.

\begin{Def}\label{def:P}
  $p\in P$ consists of the following objects, satisfying the following 
  properties:
  \begin{enumerate}
    \item $\dom(p)\subseteq I$ is countable and closed under $\varepsilon$.
    \item If $\epsilon\in \dom(p)\cap \mu$,
      then $p(\epsilon)$ is a Sacks condition.
    \item\label{item:disjointtruks}
      If $\epsilon_1\neq \epsilon_2\in \dom(p)\cap \mu$,
      then $\stem(p(\epsilon_1))$ and $\stem(p(\epsilon_2))$
      are incompatible.
    \item\label{item:defP1} If $\alpha\in \dom(p)\cap I_\epsilon$,
      then $p(\alpha)$ is a function from $p(\epsilon)$ to the
      power set of $\omega^{<\omega}$ satisfying the following:
      \begin{enumerate}
        \item  If $s\in p(\epsilon)\cap 2^n$, then $p(\alpha,s)\subseteq
          \omega^{\leq n+1}$ is an $s$-tree.
        \item 
          If $s<t$ are in $p(\epsilon)$ and $s\in 2^n$,
          then $p(\alpha,s)=p(\alpha,t)\cap \omega^{\leq n+1}$.
        \item\label{item:defP2}
          For $l\in\omega$ 
          and $s\in p(\epsilon)$ there is an
          $s'>s$ in $p(\epsilon)$ such that
          $p(\alpha,s')$ has $l$-large splitting.
      \end{enumerate}
  \end{enumerate}
\end{Def}

Note that item \ref{item:disjointtruks} is a real restriction
in the sense that $P$ is not dense in the product defined 
as above but without item \ref{item:disjointtruks}.

Item \ref{item:defP2} implies also the following seemingly stronger variant (in
\ref{def:uniform} we will use yet another one):
If $p\in P$, $\alpha\in I_\epsilon\cap \dom(p)$, $l\in\omega$ and $s\in
p(\epsilon)$, then there is an $s'>s$ in $p(\epsilon)$ such that every branch
in $p(\alpha,s')$ has $l$ many $l$-large splitting nodes.
(Any finite $s$-tree can be $l$-large for finitely many $l$ only, so we can
first extend $s$ to some $s'_0$ witnessing $l$-largeness, then to some $s'_1$
witnessing $l_1$-largeness for some sufficiently large $l_1$ etc.)

The order on $P$ is the natural one:
\begin{Def}
  For $p,q\in P$, we define $q\leq p$ by:
  \begin{itemize}
    \item $\dom(q)\supseteq \dom(p)$.
    \item If $\alpha\in\dom(p)\cap \mu$, 
      then $q(\alpha)\subseteq  p(\alpha)$.
    \item If $\alpha\in\dom(p)\cap I_\epsilon$
      and $s\in q(\alpha)\cap \omega^n$, then
      $q(\alpha,s)\subseteq p(\alpha,s)$.
  \end{itemize}
\end{Def}

\begin{Def}
  \begin{itemize}
    \item For $\alpha\in I$, $\eta_\alpha$ is the 
      $P$-name of the generic at $\alpha$.\footnote{
        More formally: If $\epsilon\in \mu$, then $\eta_\epsilon=
        \bigcup_{p\in G}\stem(p(\epsilon))$.
        \\
        If $\alpha\notin\mu$,
        then $\eta_\alpha=\bigcup\{\stem(p(\alpha,s)):\,
        p\in G,s\in \stem(p(\varepsilon(\alpha)))\}$.
        }
    \item
      $f_\epsilon: \omega\to\omega$ is the $P$-name for the function
      defined by $f_\epsilon(n)=f_{\eta_{\epsilon}\restriction n}$, and
      analogously for $g_\epsilon$.
  \end{itemize}
\end{Def}

It is straightforward to check%
\footnote{This uses e.g.\ the fact
  that for every $p\in P$, $\alpha\in I$ and $h\in\omega$ there is a $q\leq p$
  such that $\alpha\in\dom(q)$ and all stems in $q$ have height at least $h$.
  To see that \ref{def:P}.\ref{item:disjointtruks} does not
  prevent us to increase the domain, use the argument in the proof
  of \ref{lem:complete}.}
that $\leq$ is transitive and that $\eta_\alpha$
is indeed the name of an element of $\omega^\omega$.
If $\alpha\in \mu$, then $\eta_\alpha\in 2^\omega$, otherwise
$\eta_\alpha(n)<f_{\varepsilon(\alpha)}(n)$ for all $n\in\omega$.

\section{Preservation of cardinals, $\kappa_\epsilon\leq \mycfa(f_\epsilon,g_\epsilon)$}


\begin{Lem}
  $P$ is $\al2$-cc.
\end{Lem}

\begin{proof}
  Assume towards a contradiction that
  $A$ is an antichain of size $\al2$.
  Without loss of generality $\{\dom(p):\, p\in A\}$ 
  forms a $\Delta$-system with root $u\subseteq I$.
  We fix enumerations $\{\alpha^p_0,\alpha^p_1,\dots\}$ of 
  $\dom(p)$ for all $p\in A$.
  We can assume
  that the following are independent of $p\in A$
  (for $i,j\in\omega$ and $\beta\in u$):
  $p\restriction u$; 
  the statements ``$\alpha^p_i=\beta$'',
  ``$\alpha^p_i\in \mu$'', 
  ``$\alpha^p_i=\varepsilon(\alpha^p_j)$'';
  and the sequence of Sacks conditions
  $(p(\alpha^p_i):\, \alpha^p_i\in \mu)$.

  Pick elements $p,q$ of $A$. We will show $p\comp q$.
  Take $p\cup q$ and modify it the following way:
  If $i\in\omega$ is such that $\alpha^p_i\in \mu$
  and $\alpha^p_i\neq \alpha^q_i$, then we
  extend the stems 
  of (the identical Sacks conditions) $p(\alpha^p_i)$ and 
  $q(\alpha^q_i)$ in an incompatible way (e.g.\ at the first 
  split, we choose the left node for $p$ and the right one
  for $q$).
  We call the result of this $r$.
  Then $r\in P$ and $r\leq p,q$:
  Assume that $\alpha^p_i\neq \alpha^q_j$ are in $\dom(r)\cap \mu$.
  If $i\neq j$, then $q(\alpha^q_j)=p(\alpha^p_j)$ has an incompatible
  stem with $p(\alpha^p_i)$, so the (possibly longer) stems 
  in $r$ are still incompatible. If
  $i=j$, we made the stems in $r$ incompatible.
\end{proof}

\begin{Lem}\label{lem:puredec}
  $P$ has fusion and pure decision.
  In particular $P$ has continuous reading of names, and
  $P$ is
  is proper and $\omega^\omega$-bounding.
  Therefore $P$ preserves all cardinals and
  forces $2^\al0=\mu$.
\end{Lem}

The proof is straightforward, but the notation a bit
cumbersome.
\begin{Def}
  \begin{itemize}
    \item
      $\pos(p,\mathord\leq n)$ is the set of
      sequences $a=(a(\alpha))_{\alpha\in \dom(p)}$
      such that  $a(\alpha)\in\omega^{n+1}$, 
      $a(\alpha)\in p(\alpha)$
      for $\alpha\in \mu$, and $a(\alpha)\in p(\alpha,a({\varepsilon(\alpha)}))$
      otherwise. 
    \item
      For $a\in \pos(p,\mathord\leq n)$, $p\wedge a$ is the 
      result of extending the stems in $p$ to $a$.\footnote{
        More formally: $[p\wedge a](\epsilon)$ is
        $\{s\in p(\epsilon):\, s\comp a(\epsilon)\}$ for
        $\epsilon\in\mu$, and
        \\
        $[p\wedge a](\alpha,s)$ is
        $\{t\in p(\alpha,s):\, t\comp a(\alpha)\}$ for
        $\alpha\in I_\epsilon$.
        $p\wedge a$ is again a condition in $P$.
      }
    \item 
      Let $\tau$ be a $P$-name. $\tau$ is
      $(\mathord\leq n)$-decided by $p$,
      if for all $a\in \pos(p,\mathord\leq n)$, $p\wedge a$ decides $\tau$
      (i.e.\ there is some $x\in V$ such that $p\wedge a$ forces $\tau=\std x$).
    \item Assume $q\leq p$. $\pos(p,\mathord\leq n)\equiv \pos(q,\mathord\leq n)$
      means that for all $a\in \pos(p,\mathord\leq n)$ there is exactly one
      $b\in \pos(q,\mathord\leq n)$ such that $a$ is $b$ restricted to
      $\dom(p)$. In other words: On $\dom(p)$, $p$ and $q$ are 
      identical up to height $n+1$, and the stems of $q$ outside
      of $\dom(p)$ have height at least $n+1$.
      If $\dom(q)=\dom(p)$, then $\pos(p,\mathord\leq n)\equiv \pos(q,\mathord\leq n)$
      is equivalent to $\pos(p,\mathord\leq n)=\pos(q,\mathord\leq n)$.
    \item $p\in P$ is finitary if $\pos(p,\mathord\leq n)$ is finite 
      for all $n\in\omega$.
  \end{itemize}
\end{Def}

\begin{Lem}\label{lem:finite}
  The set of finitary conditions is dense in $P$.
\end{Lem}

(Enumerate $\dom(p)$ as $(\alpha_i)_{i\in\omega}$, and extend all stems at
$\alpha_i$ to height at least $i$.)

The set of finitary conditions is not open, but we get the following: If $p\in
P$ is finitary and $q\leq p$ is such that $\dom(q)=\dom(p)$, then $q$ is
finitary.

We now consider a strengthening of the property
~\ref{def:P}.\ref{item:defP2} of conditions in $P$:
\begin{Def}\label{def:uniform}
  $p$ is uniform, if for all $\alpha\in I_\epsilon$ and
  $l\in \omega$ there is a $h\in\omega$ such that
  $p(\alpha,s)$ is $l$-large for all $s\in p(\epsilon)\cap \omega^h$.
\end{Def}
First, we briefly comment on the connection between fronts
and maximal antichains in Sacks conditions:\footnote{Of course, the same
  applies to all lim-sup finite splitting tree forcings.}
Let $T$ be a perfect tree.  ``$A$ is a front'' is stronger than ``$A$ is a
maximal antichain''. In particular, it is possible that $p\in P$ is not
uniform, e.g. that for $\alpha\in I_\epsilon$ the set of nodes $s\in
p(\epsilon)$ such that $p(\alpha,s)$ has $1$-large splitting contains a maximal
antichain, but not a front. (For example, we can assume that 
$p(\epsilon)=2^{<\omega}$, $p(\alpha,0^n)$
has a trunk of length at least $n+1$, but that $p(\alpha,{0^n}^\frown 1)$ has
$1$-large splitting. So the nodes that guarantee $1$-large splitting contain
the maximal antichain $\{1,01,001,\dots\}$, but no front.)
However, if $A_1,A_2,\dots$ are maximal antichains in $T$, we can find a
perfect tree $T'\subseteq T$ such that $A_i\cap T'$ is a front in $T'$.
(Construct finite approximations $T_i$ to $T'$: For every leaf $s\in T_{i-1}$,
extend $s$ to some $s'$ above some element of
$A_i$ and further to some splitting node $s''$.
Let $T_i$ contain the successors of all these splitting nodes.)

This implies that the uniform conditions are dense:
\begin{Lem}\label{lem:wrq}
  Assume $p\in P$. Then there is a uniform $q\leq p$ such that $\dom(q)=\dom(p)$.
\end{Lem}

\begin{proof}
  Fix $\epsilon\in\mu$. Enumerate $\dom(p)\cap I_\epsilon$ as
  $\alpha_0,\alpha_1,\dots$. For
  $i,l\in\omega$ and $s\in p(\epsilon)$ and
  there is an $s'>s$ such that 
  $p(\alpha_i,s')$ has $l$-large splitting. This gives 
  (open) dense sets $D_{i,l}\subseteq p(\epsilon)$.
  Choose maximal antichains $A_{i,l}\subseteq D_{i,l}$.
  Then there is a perfect tree $q(\epsilon)\subseteq p(\epsilon)$
  such that $A_{i,l}\cap q$ is a front in $q$ for all $i,l\in\omega$.
\end{proof}

We can also fix $p$ up to some height $h$ and do the construction starting
with $h$. Then we get:

\begin{Lem}
 Assume that $p\in P$, $h\in\omega$ and that $\pos(p,\mathord\leq h)$ is finite. 
 Then there is a finitary, uniform $q\leq p$ such that
 $\dom(p)=\dom(q)$ and $\pos(p,\mathord\leq h)=\pos(q,\mathord\leq h)$.
\end{Lem}

Using this notation, we can finally prove continuous reading of names:
\begin{proof}[Proof of Lemma \ref{lem:puredec}]
  {\bf Pure decision:}
  Fix $p\in P$ finitary, $h\in \omega$ and a $P$-name $\tau$
  for an ordinal.
  We can find a finitary, uniform $q\leq p$ which
  $(\mathord \leq h)$-decides $\tau$,
  such that $\pos(p,\mathord\leq h)\equiv\pos(q,\mathord\leq h)$.

  Proof: Enumerate $\pos(p,\mathord\leq h)$ as $a_0,\dots,a_{l-1}$.
  We just strengthen each $p\wedge a_i$ to decide $\tau$
  and glue back together the resulting conditions. More
  formally:
  Set $p_0=p$. Let $0\leq i<l$. We assume that we have 
  constructed $p_i\leq p$ such that 
  $\pos(p_i,\mathord\leq h)\equiv \pos(p,\mathord\leq h)$.
  Let $b\in \pos(p_i,\mathord\leq h)$ correspond to $a_i\in \pos(p,\mathord\leq h)$,
  and find a finitary $p'\leq p_i \wedge b$ deciding $\tau$,
  so that the length of all stems are at least $h+1$.
  Define $p_{i+1}$ the following way:
  $\dom(p_{i+1})=\dom(p')$.
  \begin{itemize}
    \item If $\alpha\in\dom(p')\setminus \dom(p_i)$,
      then $p_{i+1}(\alpha)=p'(\alpha)$.
    \item If $\epsilon\in  \dom(p_i)\cap \mu$,
      then $p_{i+1}(\epsilon)=p'(\epsilon)\cup \{s\in p_i:\ s\incomp
      b(\epsilon)\}$.
    \item Assume that $\alpha\in \dom(p_i)\cap I_\epsilon$.
      If $s\in p_{i}(\epsilon)\setminus p'(\epsilon)$,
      or if $s\in p'(\epsilon)$ is incompatible with 
      $b(\epsilon)$, 
      then $p_{i+1}(\alpha,s)=p_i(\alpha,s)$. Otherwise, 
  $p_{i+1}(\alpha,s)=p'(\alpha,s)\cup \{t\in p_i(\alpha,s):\, t\incomp b(\alpha)\}$.
  \end{itemize}
  Note that $p_{i+1}\leq p_i$, 
  $\pos(p_{i+1},\mathord\leq h)\equiv \pos(p_{i},\mathord\leq h)$ and 
  $p_{i+1}\wedge b =p'$.
  Let $q\leq p_{l}$ be finitary and uniform such that 
  $\pos(q,\mathord\leq h)\equiv \pos(p_{l},\mathord\leq h)$.
  Then $q\leq p$, $\pos(q,\mathord\leq h)\equiv \pos(p,\mathord\leq h)$
  and $q\wedge b$ decides $\tau$ for each  
  $b\in \pos(q,\mathord\leq h)$.

  {\bf Fusion:}
  Assume the following:
  \begin{itemize}
    \item $p_0\geq p_1\geq \dots$ is a sequence of finitary, uniform conditions in $P$.
    \item $h_0,h_1,\dots$ is an increasing sequence of natural numbers.
    \item $\pos(p_{n+1},\mathord\leq h_n)\equiv \pos(p_n,\mathord \leq h_n)$.
    \item $u_n\subseteq \dom(p_n)$ is finite and $\varepsilon$-closed for
      $n\in\omega$. Every $\alpha\in\bigcup_{n\in\omega}\dom(p_n)$ is
      contained in infinitely many $u_i$.
    \item If $\epsilon\in u_n \cap \mu$, then the height of the front of 
      $n$-th splitting nodes in $p_n(\alpha)$ is below $h_n$ 
      (i.e. the front is a subset of $2^{\leq h_n}$).\\
      If $\alpha\in u_n \cap I_\epsilon$ and
      $s\in p_n(\epsilon)\cap \omega^{h_n}$,
      then $p_n(\epsilon,s)$ has $n$-large splitting.
  \end{itemize}  
  Then there is a canonical limit $q\leq p_i$ in $P$.

  Proof: $q(\epsilon)$ is defined by
  $\dom(q)=\bigcup_{n\in\omega}\dom(p_n)$,
  $q(\epsilon)\cap 2^{h_i+1}=p_i(\epsilon)$, and analogously
  for $q(\alpha,s)$. Pick $\alpha\in P_\epsilon$, $s\in q(\epsilon)$
  and $l\in\omega$. Pick $n>l$ such that $\alpha\in u_n$.
  Then $p_n(\alpha,s')$ has $l$-large splitting
  for some $s'\comp s$ in $p_n(\epsilon)$. 

  {\bf Continuous reading of names, $\omega^\omega$-bounding:}
  Let $\nu$ be the name of a function from $\omega$ to $\omega$ and
  $p\in P$. Then there is an increasing
  sequence $(h_i)_{i\in\omega}$ and a
  finitary $q\leq p$ which
  $(\mathord\leq h_i)$-decides $\nu\restriction h_i$ for all
  $i\in\omega$.\footnote{Or
    $\nu\restriction 2\cdot h_i$ or just $\nu(i)$ etc., 
    that does not make any difference at that stage.}

  Proof: Pick $p_0\leq p$ finitary and uniform. 
  Construct a sequence $p_0\geq p_1\geq \dots$ suitable for
  fusion the following way: Given $p_i$, find  (by
  some bookkeeping) $u_i\subseteq \dom(p_i)$,
  pick $h_i$ large enough to witness largeness of $p_i$ $u_i$,
  and then (using pure decision)
  find $p_{i+1}$ which $(\mathord\leq h_i)$-decides $\nu\restriction h_i$.

  {\bf Properness:}
  Let $\chi$ be a sufficiently large regular cardinal, and
  let $N\esm H(\chi)$ be a countable elementary submodel, $p\in P\cap N$.
  We have to show that there is a $q\leq p$ forcing 
  $\tau\in \std N$ for every $P$-name $\tau\in N$ for an ordinal.
  We can enumerate (in $V$) all the names $\tau_i$ of ordinals in $N$.
  As above, we pick an sequence $p\geq p_0\geq p_1\geq \dots$ 
  suitable for fusion such that
  $p_i\in N$ is $(\mathord\leq h_i)$-deciding $\tau_i$ (for
  the $h_i$ used for fusion).
  In $V$, we fuse the sequence to some $q\leq p$. Then $q$ is $N$-generic.

  {\bf Preservation of cardinals} follows from $\al2$-cc and properness.

  {\bf Continuum is forced to be $\mathbb\mu$}: Let $\tau$ be the name
  of a real, and $p\in P$. There is a $q\leq p$ continuously reading
  $\tau$. I.e.\ $\tau$ can be read off $q\in P$ in a recursive manner
  (using a real parameter in the ground model). The size of
  $P$ is $\mu^\al0=\mu$, so there are only $\mu$ many
  reals that can be read continuously from some $q$. On the other hand,
  the $\eta_\epsilon$ are forced to be pairwise different.
\end{proof}

\begin{Lem}
  $P$ forces that $\kappa_\epsilon\leq \mycfa(f_\epsilon,g_\epsilon)$.
\end{Lem}

\begin{proof}
  Assume the following towards a contradiction:
  $\al1\leq \lambda<\kappa_\epsilon$,
  $B_i$ ($i\in\lambda$) are $P$-names, and
  $p$ forces that $\{B_i:\, i\in\lambda\}$
  is a covering family of $(f_\epsilon,g_\epsilon)$-slaloms.

  For every $B_i$, find a maximal antichain $A_i$ of conditions
  that read $B_i$ continuously. Because of
  $\al2$-cc, $X=\bigcup_{i\in\lambda,a\in A_i}\dom(a)$
  has size $\lambda<\kappa_\epsilon$, so there is an
  $\alpha\in I_\epsilon\setminus X$. Find a $q\leq p$ and an
  $i\in \lambda$ 
  such that $q$ forces that $\eta_\alpha(n) \in B_i(n) $ for all $n$.
  Without loss of generality,
  $q$ is uniform and stronger than some $a\in A_i$, i.e.\ $q\restriction \dom(q)\setminus
  \{\alpha\}$ continuously reads $B_i$. (And $q\restriction \{\epsilon\}$
  continuously reads $\eta_\epsilon\restriction n$
  and therefore $g_\epsilon(n)$.)

  Pick some $h$ big enough such that 
  $q(\alpha,s)$ has $2$-large splitting for all $s\in q(\epsilon)\cap \omega^h$.
  Increase the stems of $q(\beta)$ for 
  $\beta\in \dom(q)\setminus \{\alpha\}$ to some height $h'>h$ to decide
  $g_\epsilon\restriction h+1$ as well as
  $B_i\restriction h+1$.
  So the resulting condition $r$ 
  decides for all $m\leq h$ the values of $B_i(m)$ and 
  $g_\epsilon(m)$.
  $B$ is the name of an $(f_\epsilon,g_\epsilon)$-slalom, and
  therefore
  $|B_i(m)|<g_\epsilon(m)$.
  Also, $r(\alpha,\eta_\epsilon\restriction h)$ 
  has a $2$-large splitting node at some $m\leq h$. 
  But that implies that there are more than $g_\epsilon(m)$ many possibilities
  for $\eta_\epsilon(m)$. So we can 
  extend the stem or $r$ at $\alpha$ and choose some 
  $\eta_\alpha(m)\notin B_i(m)$, a contradiction.
\end{proof}

\section{The complete subforcing $P_\epsilon$, $\kappa_\epsilon\geq \mycfa(f_\epsilon,g_\epsilon)$}

\begin{Def}
  $P_\epsilon\subseteq P$ consists of conditions 
  with domain in $\{\epsilon\}\cup I_\epsilon$.
\end{Def}

\begin{Lem}\label{lem:complete}
  $P_\epsilon$ is a complete subforcing of $P$,
  and also has continuous reading of names. In
  particular, $P_\epsilon$ forces $2^{\al0}=\kappa_\epsilon$.
\end{Lem}

\begin{proof}
  Continuous reading is analogous to the case of $P$.
  To see that $P_\epsilon$ is a complete subforcing, it
  is enough to show that for all 
  $p\in P$ there is a reduction $p'\in P_\epsilon$
  (i.e.\ for all $q\leq p'$ in $P_\epsilon$, $q$ and $p$ are compatible in $P$).
  Set $p'=p\restriction (\{\epsilon\}\cup I_\epsilon)$, pick
  $q\leq p'$ in $P_\epsilon$, and set
  $r=q\cup p\restriction I\setminus (I_\epsilon\cup \{\epsilon\})$.
  If $\epsilon\in\dom(p)$, then
  $r$ is a condition in $P$ (and stronger than $q$, $p$). Otherwise, it could happen
  that 
  $\stem(q,\epsilon)$ is compatible with
  $\stem(p,\epsilon')$ for some $\epsilon'\in \mu$.
  We can assume without loss of generality that
  $\stem(q,\epsilon)\supseteq \stem(p,\epsilon')$.
  Increase the stems of both $q(\epsilon)$ and $p(\epsilon')$
  to be incompatible. Then for any $\epsilon''$, 
  $\stem(q,\epsilon)$ and $\stem(p,\epsilon'')$ are incompatible
  as well.
\end{proof}

To complete the proof of the main theorem, it remains to be shown:
\begin{Lem}\label{lem2}
  $P$ forces that the $(f_\epsilon,g_\epsilon)$-slaloms in 
  $V[G_{P_\epsilon}]$ form a cover,
  in particular that $\mycfa(f_\epsilon,g_\epsilon)\leq \kappa_\epsilon$.
\end{Lem}

For the proof, we need more notation:

Let $q\in P$.
\begin{itemize}
  \item
    For $\epsilon\in\mu$, $n$ is a splitting level of $q(\epsilon)$ if there
    is some splitting node $s\in q(\epsilon)\cap \omega^n$.
    $n$ is a unique splitting 
    level if there is exactly one such $s$.
  \item 
    Let $\alpha\in I_\epsilon$. $n$ is a splitting level
    of $q(\alpha)$ if there is some $s\in q(\epsilon)\cap \omega^{n}$
    such that some $t\in q(\alpha,s)\cap \omega^n$ is a splitting
    node.
    $n$ is a unique splitting level
    of $q(\alpha)$ if there is exactly one such $s$, and if 
    moreover for this $s$ there is exactly one $t$ as well.
  \item
    $q$ has unique splitting below $h$ if for all $n<h$ there is at most one
    $\alpha\in I$ such that $n$ is splitting level of $q(\alpha)$,
    and in this case $n$ is a unique splitting level of $q(\alpha)$.

    $q$ has unique splitting if $q$ has unique splitting
    below all $h$.
  \item
    If $q$ has unique splitting below $h$, we enumerate (in increasing order) the splitting levels
    below $h$
    (for any $\alpha$) by $(m^{\text{split}}_i)_{i\in l}$
    and the corresponding $\alpha$ by $(\alpha^{\text{split}}_i)_{i\in l}$.
    If $q$ has unique splitting, we get the corresponding infinite 
    sequences.\footnote{In this case, each
      each $\alpha\in\dom(q)$ will appear infinitely often
      in the sequence $(\alpha^{\text{split}}_i)_{i\in \omega}$, to allow for sufficiently
      large splitting.}
  \item
    $q$ has unique, large splitting if it has unique splitting and if for
    $\alpha^{\text{split}}_i\notin \mu$, the splitting node $t$ of
    height $m^{\text{split}}_i$ is $i$-large.
  \item
    Let $\nu$ be a $P$-name for a sequence in $\prod_{n\in\omega} \fmax(n)$.
    $q$ rapidly reads $\nu$ below $h$ if:
    \begin{itemize}
      \item
        $q$ has unique, large splitting below $h$.
      \item If $\alpha\in I_\epsilon$, then all splits
        at $\alpha$ are higher than some split at $\epsilon$, i.e.:
        If $\alpha^\text{split}_i=\alpha$, then
        $\alpha^\text{split}_j=\epsilon$ for some $j<i$.
      \item
        $\nu\restriction m^\text{split}_i$  is
        $(\mathord\leq m^\text{split}_i)$-decided by $q$.
      \item
        If $\alpha^{\text{split}}_i\notin \mu$,
        then $\nu\restriction m^\text{split}_i$  is even 
        $(\mathord\leq m^\text{split}_i\mathord-1)$-decided.\footnote{And
            therefore $(\mathord\leq m^\text{split}_{i-1})$-decided, since
            every $\eta\in \pos(q,\mathord \leq m^\text{split}_{i-1})$
            extend uniquely to an
            $\eta'\in \pos(q,\mathord \leq m^\text{split}_{i}-1)$.}
    \end{itemize}
    $q$ rapidly reads $\nu$ if this is the case below all $h$.
\end{itemize}

If $q$ has unique splitting, then $q$ is finitary.

\begin{Lem}
  Assume that $p\in P$ and that $\nu$ is a $P$-name 
  for a sequence in $\prod_{n\in\omega} \fmax(n)$.
  Then there is a $q\leq p$ rapidly reading $\nu$.
\end{Lem}

\begin{proof}
  We use the following notion of unique extension:
  Fix $p\in P$ finitary, $m\in \omega$, and a splitting node $s$ (or 
  $(s,t)$) in $p$ of height $h>m$.\footnote{This means:
    Either $\epsilon\in\mu$ and $s\in p(\epsilon)$ is a splitting node, 
    or $\alpha\in I_\epsilon$, $s\in p(\epsilon)$ and
    $t\in p(\alpha,s)$ is a splitting node.}
  Then we can extend $p$ uniquely above $m$ up to
  $s$ (or $s,t$), i.e.\ there is a $r$ satisfying:
  \begin{itemize}
    \item $r\leq p$, $\dom(r)=\dom(p)$.
    \item $\pos(r,\mathord\leq m)=\pos(p,\mathord\leq m)$.
    \item If $m<n<h$, then $n$ is not a splitting level of $r$.
    \item $h$ is a unique splitting level of $r$.
    \item If $a\in \pos(p,\mathord\leq h)$ extends $s$ (or $s,t$),
      then $a\in \pos(r,\mathord\leq h)$.
  \end{itemize}
  In other words, we eliminate all splits between $m$ and $h$, and 
  at $h$ we leave only the split $s$ (or $t$) with all its successors.

  We use this fact to define an increasing sequence $(p_i)_{i\in\omega}$
  and show that the limit $q$ has the desired properties.

  Set $p_{-1}=p$ and $m^\text{split}_{-1}=-1$.
  Assume we already have 
  $p_i$ as well as $m^\text{split}_j$ and 
  $\alpha^\text{split}_j$ 
  for all $j\leq i$, such that $p_i$ rapidly reads $\nu$ below
  $m^\text{split}_i+1$.
  For the final limit,
  we will keep all elements of $\pos(p_i,\mathord\leq m^\text{split}_i+1)$.

  We use some bookkeeping to choose $\alpha\in \dom(p_i)$ and
  $s\in p_i(\varepsilon(\alpha))\cap \omega^{m^\text{split}_i+1}$.
  If $\alpha\in \mu$, we pick some splitting node $s'>s$ in $p_i(\alpha)$.
  Otherwise
  we again use the bookkeeping to choose $t\in p_i(\alpha,s)\cap \omega^{m^\text{split}_i+1}$,
  and pick some
  $s'>s$ in $p_i(\varepsilon(\alpha))$ and an $i+2$-big splitting
  node $t'>t$ 
  in $p_i(\alpha,s')$. 
  Let $h$
  be the height of the splitting node $s'$ (or $t'$).
  We extend $p_i$ uniquely above $m^\text{split}_i$
  to $s'$ (or $s',t'$). Call the result $r$.
  Set $m^\text{split}_{i+1}=h$.
  Then, using pure decision,
  we can find some $p'\leq r$ which is
  $(\mathord\leq h)$-deciding
  $\nu\restriction h$ so that
  $\pos(p',\mathord\leq h)\equiv \pos(r,\mathord\leq h)$
  and the stems of $p'$ outside of $\dom(r)$ are higher
  than $h$.

  If $\alpha\in\mu$, set $p_{i+1}=p'$. Otherwise,
  let $A$ be the set of successors of $t'$.
  There are 
  less than $\fmax(h-1)^h$
  many possibilities for $\nu\restriction h$,
  and at most $h$ many splitting nodes below $h$,
  each with at most $\fmax(h-1)$ many successors.
  This gives a function
  \[
     \fmax(h-1)^h\times A\to \fmax(h-1)^h
  \]
  or 
  \[
    A \to \fmax(h-1)^{h\cdot \fmax(h-1)^h}<\gmin(h).
  \]
  So we can use bigness to thin out $A$ to some homogeneous
  $B$ that has norm at least $i+1$. Call
  the result $p_{i+1}$. In this case.
  $p_{i+1}$ already $(\mathord\leq h\mathord-1)$-decides
  $\nu\restriction h$.

  Let $q$ be the limit of $(p_i)_{i\in\omega}$. We have 
  to show that $q\in P$. It is enough to require from the bookkeeping
  that the following is satisfied:
  \begin{itemize}
    \item For all $\epsilon\in\dom(q)\cap \mu$, and $s_0\in q(\epsilon)$,
      there is an $s>s_0$ such that the bookkeeping chooses $\epsilon,s$
      at some stage.
    \item For all $\alpha\in \dom(q)\cap I_\epsilon$, for all
      $s_0\in q(\epsilon)$, and for all $t_0\in q(\alpha,s_0)$,
      there are $s>s_0$ and $t>t_0$ such that $\alpha,s,t$ are
      chosen at some stage.
    \item  For all $\alpha\in \dom(q)\cap I_\epsilon$, 
      $\epsilon$ is chosen (for the first time) before 
      $\alpha$ is chosen.
  \end{itemize}
  (It is easy to find a bookkeeping meeting these requirements.)
  Then $q$ is indeed in $P$: Assume that $\alpha\in \dom(q)\cap I_\epsilon$,
  $s_0\in q(\epsilon)$, and $l\in\omega$.
  We have to show that $q(\alpha,s)$ 
  is $l$-large for for some $s>s_0$.
  First extend $s$ to some $s'$ of height at least $m^\text{split}_l$
  (defined from $q$). Enumerate the leaves in $q(\alpha,s')$ as
  $t^{0},t^1,\dots, t^{k-1}$.
  Increase $s'$ to $s'_0$ such that in $q(\alpha,s'_0)$  there
  is a splitting node above $t^0$. Repeat that for the other 
  $t^i$ and set $s=s'_{k-1}$.
  If $b$ is a branch through 
  $q(\alpha,s)$, then there has to be some split in $b$ above
  $m^\text{split}_l$, but each splitting node 
  in $q$ of this height is $l$-large.
\end{proof}

So we get: If $\alpha^\text{split}_{i+1}\notin \mu$, then $\nu\restriction
m^\text{split}_{i+1}$, and in particular $\nu(m^\text{split}_{i})$,
is $(\mathord \leq m^\text{split}_i)$-decided.
Otherwise, it is $(\mathord \leq m^\text{split}_i)$-decided 
only modulo the two possibilities left and right for the successor at the
split at height $m^\text{split}_{i+1}$ in the Sacks condition
$q(\alpha^\text{split}_{i+1})$.
So in both cases, and for all $n$, we can calculate $\nu(n)$ from
$2\times \pos(q,\mathord\leq n)$. We can write this as a function:
\[
  G: 2\times \pos(q,\mathord\leq n)\to \fmax(n).
\]

\begin{proof}[Proof of Lemma \ref{lem2}]
  Fix $p\in P$ and a $P$-name $\nu$ for a function in
  $\prod_{n\in\omega}f_\epsilon (n)$. We have to find $q\leq p$
  and a $P_\epsilon$-name $B$ of an $(f_\epsilon,g_\epsilon)$-slalom 
  such that $q$ forces $\nu(n)\in B(n)$ for all $n\in\omega$.

  Let $r\leq p$ rapidly read $\nu$. We can assume that $\epsilon\in\dom(r)$.
  We can also assume that the $i$-th splitting node is even $(i+1)$-large
  and not just $i$-large.\footnote{It is clear we can get this
    looking at the proof of rapid reading, or we can get first a ``standard''
    rapid reading $r$ and then just remove the very first split by enlarging
    the trunk.}
  We will define, by induction on $n$, 
  $B(n)$ as well as $q\leq r$ up to height $\mathord\leq n$.

  $q$ will be the result of thinning out some of the splitting nodes in $r$ (in
  the non-Sacks part), in a such way
  that the norm of the node will be decreased by at most 1.
  So $q$ will again have unique, large splitting, and $q$
  will be a condition in $P$.

  If we already constructed $q$ below $n$, and if there is no split at
  height $n$, we have no choice for $q$ at height $n$ but just take the unique
  extension given by $r$.  If there is a split, we may thin out the successor set
  (reducing the norm by at most $1$).  Of course, this way we will loose former
  splits at higher levels (which extended the successors we just left out). So the
  splitting levels of $q$ will be a proper subset of the splitting levels of $r$. 
  In the following,
  $m^\text{split}_i$ and $\alpha^\text{split}_i$ denote the splits of $q$.

  If $\epsilon'\neq \epsilon$,
  $\alpha\in\dom(r)\cap I_{\epsilon'}$, and $h$ is a splitting level
  of $r(\alpha)$, then there is some splitting level $h'<h$ of $r(\epsilon')$.
  Also, $\trunk(r,\epsilon)$ and $\trunk(r,\epsilon')$ are incompatible,
  i.e. they differ below $h$. By the way we construct $q$, we get the
  same for $q$:
  \begin{quote} $(*)$ If $\alpha\in I_{\epsilon'}$, $\epsilon'\neq \epsilon$, and if 
    $h$ is a splitting level of $q(\alpha)$, then
    either all $s\in q(\epsilon)\cap 2^h$ are lexicographically smaller 
    than all $t\in q(\epsilon')\cap 2^h$, or the other way round.
  \end{quote}

  We now define $q$ at height $n$ and $B(n)$:
  Assume that $i$ is maximal such that
  $m=m^\text{split}_i\leq n$.
  Set $\alpha=\alpha^\text{split}_i$.
  By rapid reading there is a function $G$ with domain
  $2\times \pos(r,\mathord\leq m)$ that calculates
  $\nu(n)$. 
  Let $A$ be the set of successors of the split of level $m$.
  $\pos(r,\mathord\leq m-1)$ has size at most $\fmax(m-1)^m$.
  So we can write $G$ as 
  \[
    G: 2\times \fmax(m-1)^m \times A \to f_\epsilon(n).
  \]

  {\bf Case A:} $n>m$.\\
    There are no splits on level $n$, so for $q$ at level $n$ we use the
    unique extensions given by $r$.\\
    The size of $A$ is at most $\fmax(m)$, so the
    domain of $G$ has at most size 
    \[ 
      2\cdot \fmax(m-1)^m\cdot \fmax(m)<\gmin(n),
    \]
    and therefore is smaller than $g_\epsilon(n)$. So we can put all
    possible values for $\nu(n)$ into $B(n)$.

  {\bf Case B:} $n=m$, $\alpha\in \{\epsilon\}\cup I_\epsilon$.\\
    $q$ at level $n$ contains all the successors of the split at level $n$.\\
    In the $P_\epsilon$-extension, we know which 
    successor we choose.\footnote{If any. Of course
    the filter could be incompatible with $s$ (or $s,t$).}
    Given this knowledge, the domain of $G$ is again smaller than 
    $\gmin(m)$, just as in Case~A.

  {\bf Case C:} $n=m$, $\alpha\in \mu\setminus\{\epsilon\}$.\\
    $q$ at level $n$ contains both successors of the split at level $n$.\\
    $|A|=2$, so there are again only
    \[
      2\cdot \fmax(n-1)^{n}\cdot 2<\gmin(n)
    \]
    many possible values for $\nu(n)$.

  {\bf Case D:} Otherwise
                $n=m$, $\alpha\in I_{\epsilon'}$, $\epsilon'\neq \epsilon$.\\
    So for an $s\in r(\epsilon')\cap \omega^n$ there
    is a splitting node $t\in r(\alpha,s)$ of height $n$
    with successor set $A$.
    As stated in $(*)$ above, 
    $s$ is (lexicographically) either smaller or larger
    than all the nodes in $r(\epsilon)\cap \omega^n$.

  {\bf Subcase D1:} $s$ is smaller.\\
    We keep all the successors of the split at level $n$.\\
    $|A|\leq f_s$, and
    $g_\epsilon(n)=g_{\eta_\epsilon\restriction n}$ 
    has to be some $g_{n,k}$ for 
    $k>s$ (in $[0,2^n-1]$).
    So we get 
    \[
      2\cdot \fmax(n-1)^{n}\cdot f_s<g_\epsilon(n)
    \]
    many possible values.

  {\bf Subcase D2:} $s$ is larger.\\
    Let $k$ be $s-1$ (in $[0,2^n-1]$). 
    So $\nu(n)$ is less than $f_{n,k}$.
    We can transform $G$ into a function
    \[
      F: A \to f_{n,k}^{2\cdot \fmax(n-1)^{n}}<g_{n,s}.
    \]
    So we can thin out $A$ to get an $F$-homogeneous set $B\subseteq A$,
    decreasing the norm by at most 1.
    $q$ at height $n$ contains only the successors in $B$.
    Modulo $q$, there remain only $2\cdot \fmax(n-1)^{n}$ many possibilities
    for $\nu(m)$.
\end{proof}

\bibliographystyle{plain}
\bibliography{listb,more}

\end{document}